\documentclass{article}

\usepackage{amsmath}
\usepackage{amssymb}
\usepackage{amsthm}

\newtheorem{theorem}{Theorem}[section]
\newtheorem{lemma}[theorem]{Lemma}
\newtheorem{corollary}[theorem]{Corollary}
\newtheorem{proposition}[theorem]{Proposition}
\theoremstyle{definition}
\newtheorem{definition}[theorem]{Definition}

\newtheorem{question}[theorem]{Question}

\newcommand\defined{\mathord\downarrow}

\newcommand\pair[2]{\left\langle #1, #2 \right\rangle}
\newcommand\trunc{\mathop{\upharpoonright}}

\newcommand\leqT{\leq_{\mathrm T}} % Turing reducibility
\newcommand\geqT{\geq_{\mathrm T}}
\newcommand\join\oplus

\newcommand\symmdiff{\mathop\triangle}
\newcommand\emptystring\emptyset

\usepackage[colorlinks]{hyperref}
\begin{document}

\title{Kolmogorov Complexity of \\ Attractive Degrees}
\author{Tiago Royer}
\maketitle

\begin{abstract}
    This paper proves a Kolmogorov-complexity-flavored sufficient condition for a set to be attractive
    (Theorem~\ref{thm:2Kn}),
    and discusses some consequences of this condition.
\end{abstract}

\section{Introduction}

The asymptotic density of a set $A \subseteq \mathbb N$
is the number
\begin{equation}
    \limsup_n \frac{|A \cap [0, n)|}{n}.
    \label{eq:asymptotic-density}
\end{equation}
We say that $A$ is \emph{sparse} if this number is zero,
and that $A$ and $B$ are \emph{coarsely similar}
if their symmetric difference $A \symmdiff B$ is sparse.

This asymptotic density can be used to define a pseudometric on $2^{\mathbb N}$
by letting $\delta(A, B)$ be the asymptotic density of the symmetric difference $A \symmdiff B$.
For a set $A$ and a Turing degree $\mathbf b$,
we use $\delta$ to define $\gamma_{\mathbf b}(A)$ as
\begin{equation*}
    \gamma_{\mathbf b}(A) = \sup_{B \in b} 1 - \delta(A, B).
\end{equation*}
This is a measure of how difficult $A$ is to compute,
from the point of view of $B$.
For example,
if $A$ is coarsely computable relative to $\mathbf b$,
then $\gamma_{\mathbf b}(A) = 1$
(``$\mathbf b$ can tell membership in $A$ asymptotically all the time''),
and if $\mathbf b$ is $1$-random relative to $A$
then $\gamma_{\mathbf b}(A) \geq 1/2$
(``$\mathbf b$  can guess at least half of the entries of $A$'').
We can ``lift'' this definition to Turing degrees
by taking infimums:
\begin{equation*}
    \Gamma_{\mathbf b}(\mathbf a) = \inf_{\substack{A \in \mathbf a \\ B \in \mathbf b}} \gamma_B(A).
\end{equation*}
(We also write $\Gamma_B(A)$ for $\gamma_{\mathbf b}(\mathbf a)$
where $A \in \mathbf a$ and $B \in \mathbf b$.)

Intuitively,
this asks $\mathbf b$ what is the hardest challenge that $\mathbf a$ can pose to it,
with respect to coarse computations.
In fact,
we can show that $\Gamma_{\mathbf b}(\mathbf a) = 1$
if and only if $\mathbf b \geq \mathbf a$.

This fact can be used to define a metric on Turing degrees,
called the Hausdorff distance $H(\mathbf a, \mathbf b)$ between $\mathbf a$ and $\mathbf b$:
\begin{equation*}
    H(\mathbf a, \mathbf b) = 1 - \min \{\Gamma_{\mathbf b}(\mathbf a), \Gamma_{\mathbf a}(\mathbf b)\}.
\end{equation*}
(The name ``Hausdorff distance'' is justified
because we can show that $H(\mathbf a, \mathbf b)$
equals the classical Hausdorff distance between the closures of $\mathbf a$ and $\mathbf b$
in the pseudometric space $(2^{\mathbb N}, \delta)$
\cite[Proposition~4.15]{metricpdf}.)

The ``$1-$'' is here because $\Gamma_{\mathbf b}(\mathbf a)$ is a ``measure of closeness'',
where $\Gamma_{\mathbf b}(\mathbf a) = 0$ if $\mathbf a$ is ``hard'' for $\mathbf b$,
and $\Gamma_{\mathbf b}(\mathbf a) = 1$ if it is ``easy''.
In fact,
$H$ is a metric on the space of Turing degrees.
(Analogously we write $H(A, B)$ for $H(\mathbf a, \mathbf b)$
where $A \in \mathbf a$ and $B \in \mathbf b$,
but for sets the function $H$ is just a pseudometric.)

In~\cite{Monin2018},
Monin showed that $\Gamma_{\emptyset}(A)$ can only attain the values $0$, $1/2$, and $1$,
and \cite[Theorem~4.20]{metricpdf} extended this result to all $\Gamma_B(A)$.
In fact,
we get the following characterization of when $\Gamma_B(A)$ attains each value.

\begin{definition}
    A function $f$ is \emph{$2^{2^n}$\!-infinitely often equal} (i.o.e.) relative to $A$
    if, for every function $g \leqT A$ such that $g(n) < 2^{2^n}$ for all $n$,
    there exists infinitely many $n$ such that $f(n) = g(n)$.
    \footnote{
        We could replace the $2^{2^n}$ with any function $F$
        to obtain the definition of $F$-infinitely often equal,
        but the only bound we will use in this paper is $F(n) = 2^{2^n}$.
    }
\end{definition}

\begin{theorem}[\cite{Monin2018,metricpdf}]
    The following are equivalent.
    \begin{itemize}
        \item $\Gamma_B(A) > 1/2$.
        \item $\Gamma_B(A) = 1$.
        \item $B \geqT A$.
    \end{itemize}
\end{theorem}

\begin{theorem}[\cite{Monin2018,metricpdf}]
    The following are equivalent.
    \begin{itemize}
        \item $\Gamma_B(A) < 1/2$.
        \item $\Gamma_B(A) = 0$.
        \item $A$ computes a function $f$ which is $2^{2^n}$ infinitely often equal relative to $B$.
    \end{itemize}
    \label{thm:gamma-equals-0}
\end{theorem}
(Note that $A$ is the one computing the function,
even though that $B$ is being used as an oracle to try to approach $A$.
Intuitively,
the set $A$ can present $f$ as the ``challenge'',
which is ``too difficult'' for $B$ to compute.)

This immediately implies that $H(\mathbf a, \mathbf b) \in \{0,1/2,1\}$.

We now define the main object of study of this paper.

\begin{definition}[{\cite[Definition 5.10]{metricpdf}}]
    A Turing degree $\mathbf d$ is called \emph{attractive}
    if there are measure-1 many degrees $\mathbf c$
    such that $H(\mathbf d, \mathbf c) = 1/2$,
    and is called \emph{dispersive}
    if there are measure-1 many degrees $\mathbf c$
    such that $H(\mathbf d, \mathbf c) = 1$.
\end{definition}

Kolmogorov's 0-1 law immediately shows that every degree is either attractive or dispersive.

\begin{proposition}[{\cite[Proposition 5.11]{metricpdf}}]
    The class of attractive degrees is closed upwards,
    and the class of dispersive degrees is closed downwards.
    \label{thm:closure-properties}
\end{proposition}

\begin{proof}
    A useful trick is to note that,
    for any noncomputable set $A$,
    there are measure-1 many sets $B$ which are $1$-random relative to $A$,
    and which do not compute $A$.
    For any such set $B$,
    being random implies that,
    for every $A$-computable set $\hat A$,
    we have $\delta(B, \hat A) = 1/2$,
    thus $\Gamma_B(A) \geq 1/2$;
    and $B \not\geq_T A$ implies $\Gamma_B(A) < 1$,
    so for every set $A$ there are measure-1 many sets $B$ for which $\Gamma_B(A) = 1/2$.

    Hence,
    to show that $A$ is attractive,
    it suffices to show that there are measure-1 many sets $B$ for which $\Gamma_A(B) = 1/2$;
    conversely,
    to show that $A$ is dispersive,
    we must show that there are measure-1 many sets $B$ for which $\Gamma_A(B) = 0$.

    So,
    let $A$ be attractive,
    and $C \geq_T A$.
    For any $B$ such that $\Gamma_A(B) = 1/2$
    we then also have $\Gamma_C(B) \geq 1/2$.
    There exists measure-1 many such $B$.
    Discarding the measure-0 many such $B$ that might be computable by $C$,
    we get measure-1 many sets $B$ for which $\Gamma_C(B) = 1/2$,
    which (as the above observation shows)
    implies that $C$ is also attractive.

    That dispersive sets are closed downwards
    follows by the fact that the two classes are complementary.
\end{proof}

\section{Attractive Degrees}

We can use Theorem~\ref{thm:gamma-equals-0}
to prove the following sufficient condition for a degree to be attractive.

\begin{theorem}
    Let $A$ be a set such that,
    for some constant $C$ and all large enough $n$,
    we have $K(A \trunc 2^n) > 2K(n) - C$.
    Then $A$ is attractive.
    \label{thm:2Kn}
\end{theorem}

That is,
$K(A \trunc 2^n)$ may dip below $2K(n)$,
but not by much.
Observe that this is implied by the stronger condition
``$K(A \trunc n) > 2K(n) - C$ for all large enough $n$'',
because $K(2^n) = K(n) + O(1)$.

\begin{proof}
    As noted in the proof of Proposition~\ref{thm:closure-properties},
    it suffices to show that there are measure-1 many sets $B$ such that $\Gamma_A(B) = 1/2$.
    For $B$ to fail to satisfy $\Gamma_A(B) = 1/2$,
    unless it is one of the measure-0 many sets below $A$
    (in which case we'd have $\Gamma_A(B) = 1$),
    it must be the case that $\Gamma_A(B) = 0$;
    i.e.,
    the set $B$ has $2^{2^n}$-infinitely often equal degree relative to $A$.
    This means that there exists a single function $\Phi_e^B$
    such that,
    for every $A$-computable function $f$ bounded by $2^{2^n}$,
    there exists infinitely many $n$ with $f(n) = \Phi_e^B(n)$.

    Fix a Turing functional $\Phi_e$,
    and let $\mathcal I$ be the class of sets $B$
    such that $\Phi_e^B$ is an i.o.e.\ function as above.
    We will show that the measure of $\mathcal I$ is zero.
    For the sake of contradiction,
    assume that this is not the case.
    Let $f(n)$ be the first $2^n$ bits of $A$,
    interpreted as the binary expansion of a number.
    We will find concise descriptions of $f(n)$ (hence of $A \trunc 2^n)$)
    as follows.

    For each $n$ and $m$,
    let $\mu(n,m)$ be the measure of all the $B$ such that $\Phi_e^B(n) \defined = m$,
    and let $\kappa(n,m)$ be the smallest $k$ such that $2^{-k} < \mu(n,m)$.
    We would like to construct a sequence of KC requests of the form $(K(n) + \kappa(n,m), m)$,
    but the numbers $K(n) + \kappa(n,m)$ are not computable.
    So we approximate them from above as follows.

    For each $n$ and each $m < 2^{2^n}$,
    simultaneously compute $\Phi_e^\sigma(n)$ for all $\sigma$,
    and search for (Kolmogorov) descriptions $\tau$ of $n$.
    Take note of the measure of all the $\sigma$ for which $\Phi_e^\sigma(n) \defined = m$.
    If this measure surpasses $2^{-k}$ for some integer $k$,
    issue the request $(k + |\tau|, m)$,
    where $\tau$ is the shortest description of $n$ found so far.
    Similarly,
    if a shorter description $\tau$ is found for $n$,
    issue the request $(k + |\tau|, m)$
    for the same $k$
    (i.e. $k$ is the smallest integer such that
    $2^{-k}$ is a lower bound to the measure of all $\sigma$ for which $\Phi_e^\sigma(n)$
    has been observed to converge to $m$).

    First,
    let us show that this is indeed a KC set.
    The process of searching for better Kolmogorov descriptions of $n$
    means that,
    eventually,
    we will issue the request $(k + K(n), m)$.
    All other similar requests $(k + |\tau|, m)$ will have exponentially smaller weight;
    for a fixed $k,m,n$,
    the sum of the weights $2^{-k-|\tau|}$ for all found descriptions $\tau$ of $n$
    is thus bounded by $2^{-k - K(n) + 1}$.
    (Note that this bound is true even if $k$ changes before we find an optimal $\tau$,
    as $\tau$ is never shorter than $K(n)$.)
    Similarly,
    for each $n, m$,
    as we observe progressively more $\sigma$ with $\Phi_e^\sigma(n)$ converging to $m$,
    we get a progressively better approximation to $\mu(n,m)$,
    so eventually the request $(\kappa(n,m) + K(n), m)$ will be issued.
    By a similar analysis,
    the sum of the weights $2^{-k - K(n) + 1}$ of requests issued for a fixed $m, n$
    is bounded by $2^{-\kappa(n,m) - K(n) + 2}$.
    For a fixed $n$,
    the numbers $\mu(n, m)$ measure disjoint sets,
    so their sum is at most $1$.
    Because $2^{-\kappa(n,m)} < \mu(n,m)$,
    the sum of the $2^{-\kappa(n,m)}$ is also at most $1$.
    Finally,
    the fact that $\sum_n 2^{-K(n)} < 1$ shows
    that the collection of requests above is indeed a KC set.
    Summarizing,
    we have
    \begin{align*}
        \sum_{m,n,k,\tau : \text{request $(k+|\tau|, m)$}} 2^{-k-|\tau|}
            &\leq \sum_{m,n,k : \text{request $(k+K(n), m)$}} 2^{-k-K(n)+1} \\
            &\leq \sum_{m,n} 2^{-\kappa(n,m) - K(n) + 2} \\
            &= 4 \sum_n 2^{-K(n)} \sum_m 2^{-\kappa(n,m)} \\
            &\leq 4 \sum_n 2^{-K(n)} \sum_m \mu(n,m) \\
            &\leq 4 \sum_n 2^{-K(n)} \\
            &\leq 4.
    \end{align*}

    As an immediate consequence,
    we have $K(m) \leq \kappa(n,m) + K(n) + O(1)$;
    in particular,
    we have $K(f(n)) \leq \kappa(n, f(n)) + K(n) + O(1)$,
    so (because $f(n)$ is $A \trunc 2^n$ interpreted as a number)
    there exists a fixed constant $C_1$ such that
    $K(A \trunc 2^n) \leq \kappa(n, f(n)) + K(n) + C_1$.
    We now argue that there are infinitely many $n$
    such that $\kappa(n, f(n)) \leq K(n) - C - C_1$.

    If this is not the case,
    then for large enough $n$,
    say $n > N$,
    we have $\kappa(n, f(n)) > K(n) - C - C_1$.
    In particular,
    \begin{equation*}
        \sum_{n > N} \mu(n, f(n)) \leq 2 \sum_{n > N} 2^{-\kappa(n, f(n))}
            \leq 2 \sum_{n > N} 2^{-K(n) + C + C_1}.
    \end{equation*}
    Because $\sum_n 2^{-K(n)} \leq 1$,
    by making $N$ large enough
    we can make the rightmost term of this inequality as small as we want;
    in particular,
    we can make it smaller than the measure of $\mathcal I$
    (the class of sets $B$ such that $\Phi_e^B$ is an i.o.e.\ function).
    However,
    the leftmost term of this inequality,
    $\sum_{n > N} \mu(n, f(n))$,
    is (an upper bound on) the measure of the collection of all $B$
    such that $\Phi_e^B(n) = f(n)$ for some $n > N$.
    This contains all sets of $\mathcal I$,
    which is a contradiction.

    Hence,
    there are infinitely many $n$ such that $\kappa(n, f(n)) \leq K(n) - C - C_1$,
    which means $K(A \trunc 2^n) \leq 2K(n) - C$.
    Finally,
    this contradicts our initial hypothesis
    that $K(A \trunc 2^n) > 2 K(n) - C$ for all large enough $n$,
    which means that $A$ is attractive.
\end{proof}

The bound $2 K(n) - C$ is pretty much optimal for this proof strategy.
The KC requests are (ultimately) of the form $K(n) + \kappa(n, m)$,
and we need the facts that $\sum_n 2^{-K(n)}$ and $\sum_m 2^{-\kappa(n,m)}$ are both finite.

A priori,
we could replace the $K(n)$ with any other function $F$ such that $\sum_n 2^{-F(n)} < \infty$,
but, in order to actually construct the KC set,
we need to be able to approximate $F(n)$ from above.
In other words,
the set $\{(n, k) : F(n) \leq k\}$ is a c.e.\ set,
and we can show that any such $F$ beats $K$ by at most an additive constant
(see e.g.\ \cite[Theorem~3.7.8]{DowneyHirschfeldt}).

We do have a bit of lee-way in the second term, $\kappa(n,m)$.
We bound $\kappa(n, f(n))$ by $F(n) = K(n) - C - C_1$,
and again we need the property that $\sum_n 2^{-F(n)} < \infty$.
But this part of the argument does not need to be computable,
so there are no further restrictions on $F$.
That is,
the bound $2K(n) - C$ could be replaced by $K(n) + F(n) - C$
where $F$ is any function satisfying $\sum_n 2^{-F(n)}$.
We can still rephrase this in terms of Kolmogorov complexity
by using $F$ as an oracle:
if we simply let $X$ encode the set $\{\pair{n}{F(n)}: n \in \mathbb N\}$,
then $K^X(n) \leq F(n) + O(1)$,
so the bound on $K(A \trunc n)$ becomes $K(n) + K^X(n) - C$.
We note this improvement in a corollary.

\begin{corollary}[of the proof]
    For any sets $X$ and $A$ and any constant $C$,
    if $K(A \trunc 2^n) \geq K(n) + K^X(n) - C$ for all large enough $n$,
    then $A$ is attractive.
    \qed
\end{corollary}

The rest of this section discusses some interesting consequences of this theorem.

\subsection{Genericity and Randomness}

It is known that every weakly $2$-generic set is dispersive~\cite[Theorem~5.16]{metricpdf},
and that every $1$-random set is attractive~\cite[theorem~5.5]{metricpdf}.

$1$-random sets are characterized by having $K(A \trunc n) \geq n - O(1)$;
i.e.\ all sets with high Kolmogorov complexity are attractive.
Theorem~\ref{thm:2Kn} significantly lowers what ``high'' means in this case,
thus narrowing the gap between genericity and randomness.

Interestingly,
it is still open whether every $1$-generic set is dispersive.
$1$-generic sets can only compute infinitely often $K$-trivial sets
% TODO: citation
(i.e.\ sets $B$ such that $K(B \trunc n) = K(n) + O(1)$ for infinitely many $n$),
which violate the $2K(n)$ bound from the theorem.

\subsection{Hausdorff Dimension of the class of dispersive sets}

It is known that the class of dispersive sets has measure 0~\cite[Corollary~5.8]{metricpdf}.
Using the point-to-set principle~\cite[Theorem~1]{LL2018},
we can show the following stronger result.

\begin{proposition}
    Let $\mathfrak D$ be the class of dispersive sets.
    Then the (classical) Hausdorff dimension of $\mathfrak D$ is zero.
\end{proposition}

\begin{proof}
    The point-to-set principle states that
    \begin{equation*}
        \dim_H(\mathfrak D) = \min_{A \subseteq N} \sup_{B \in \mathfrak D} \dim^A(B).
    \end{equation*}

    $\dim_H$ is the classical Hausdorff dimension,
    and $\dim^A$ is the (relativized) effective Hausdorff dimension.

    By the contrapositive of Theorem~\ref{thm:2Kn},
    for every dispersive set $B$
    there are infinitely many $n$ such that $K(B \trunc n) \leq 2K(n)$.
    This means that
    \begin{equation*}
        \dim(B) = \liminf_n \frac{K(B \trunc n)}{n} = 0.
    \end{equation*}

    So picking $A = \emptyset$ in the point-to-set principle,
    we get
    \begin{align*}
        \dim_H(\mathfrak D) &\leq \sup_{B \in \mathfrak D} \dim^A(B) \\
                            &\leq \sup_{B \in \mathfrak D} \dim(B) \\
                            &= 0.
    \end{align*}
    Hence $\mathfrak D$ has Hausdorff dimension 0.
\end{proof}

\subsection{Minimal Degrees}

Because the class of dispersive degrees is closed downwards,
using the following lemma we can show that there exists minimal dispersive degrees.

\begin{lemma}[{\cite[Theorem 5.15]{metricpdf}}]
    Every low c.e. set is dispersive.
\end{lemma}

\begin{proposition}
    There exists a dispersive set $A$ which is minimal for Turing degrees.
\end{proposition}

\begin{proof}
    Let $C$ be a low c.e. set.
    By the lemma,
    the set $C$ is dispersive.
    Now,
    by \cite[Theorem XI.4.9]{Odifreddi1999},
    the set $C$ computes a set $D$ which has minimal Turing degree.
    And because dispersive degrees are closed downwards,
    the set $D$ is also dispersive.
\end{proof}

This argument does not work for attractive degrees,
which are closed upwards.

It is known that there are sets $A$ with minimal degree and with effective packing dimension $1$
(i.e.\ $\limsup_n K(A \trunc n)/n = 1$)
\cite[Theorem~1.1]{DG2008}.
However,
we would need a lower bound on $K(A \trunc n)$ which is always true,
rather than just infinitely often.
The problem of finding a set with minimal degree and Hausdorff dimension $1$
(i.e.\ $\liminf_n K(A \trunc n)/n = 1$)
is still open,
but we can formulate the following easier question.

\begin{question}
    Are there any sets $A$ with minimal Turing degree
    which also satisfy $K(A \trunc n) > 2K(n)$ for all $n$?
\end{question}

\bibliographystyle{plainurl}
\bibliography{bib}

\end{document}